\documentclass[11pt,reqno]{amsart}

\usepackage{amsmath,amssymb,amsthm}
\usepackage[margin=1in]{geometry}

\newtheorem{theorem}{Theorem}[section]
\newtheorem{lemma}[theorem]{Lemma}
\newtheorem{corollary}[theorem]{Corollary}

\theoremstyle{definition}
\newtheorem{definition}[theorem]{Definition}
\newtheorem{example}[theorem]{Example}
\newtheorem{remark}[theorem]{Remark}

\newcommand{\D}{\mathbb{D}}
\newcommand{\C}{\mathbb{C}}
\newcommand{\N}{\mathbb{N}}
\newcommand{\Real}{\operatorname{Re}}

\begin{document}

\title[A Zalcman--Pang type rescaling result\dots]
{A Zalcman--Pang type rescaling result for $\varphi$-normal harmonic mappings and applications}

\author{Kuldeep Singh Charak}
\author{Pratiksha}
\author{Nikhil Bharti$^*$}
\thanks{$^*$ Corresponding author: Nikhil Bharti}



\begin{abstract}
We investigate normality and $\varphi$-normality criteria for harmonic mappings in the
unit disk. A new sufficient condition for normality involving extended spherical
derivatives is established. We further prove a Zalcman--Pang type rescaling lemma for
$\varphi$-normal harmonic mappings and derive new $\varphi$-normality criteria as
applications. In addition, we introduce $\varphi$-normal families of harmonic mappings
and obtain corresponding Lappan-type characterizations. In particular, we show that for
sense-preserving harmonic mappings the relevant test set may be taken to consist
of only three points.
\end{abstract}

\maketitle

\section{Introduction and main results}

The theory of normal functions and normal families occupies a central position in
geometric function theory and complex analysis. Originating in the work of Montel
\cite{Montel}, normality has become one of the most fundamental compactness principles in
the study of analytic and meromorphic functions. A family of meromorphic functions
defined in a domain $D\subseteq\C$ is said to be \emph{normal} if every sequence in the
family contains a subsequence which converges locally uniformly, with respect to the
spherical metric, to a meromorphic function or to the point at infinity. One of the most
important characterizations of normality is Marty's theorem \cite[Theorem~4]{Marty}, which
asserts that a family of meromorphic functions is normal if and only if the corresponding
spherical derivatives are locally uniformly bounded. Recall that the spherical derivative
of a meromorphic function $f$ is
\begin{equation}\label{eq:sphmero}
f^{\#}(z)=\frac{|f'(z)|}{1+|f(z)|^{2}}.
\end{equation}

Closely related to Marty's criterion is the notion of a \emph{normal function}, introduced
implicitly by Yosida \cite{Yosida} and Noshiro \cite{Noshiro}, and later formalized by
Lehto and Virtanen \cite{LehtoVirtanen}. A meromorphic function on the unit disk $\D$ is
called normal if its post-composition with conformal automorphisms of $\D$ forms a normal
family. This admits the following equivalent quantitative characterization.

\begin{definition}\label{def:normalmero}
A meromorphic function $f$ in the unit disk $\D=\{z\in\C:|z|<1\}$ is called \emph{normal} if
\[
\sup_{z\in\D}\,(1-|z|^{2})\,f^{\#}(z)<\infty.
\]
\end{definition}

In recent decades considerable efforts have been devoted to extending classical results from
analytic and meromorphic function theory to harmonic mappings driven by the enormous applicability potential of harmonic mappings (see \cite{ACA, OCM}). Recall that a complex-valued harmonic mapping defined in a simply connected domain $D\subseteq\C$ admits a canonical
decomposition $f=h+\bar g$, where $h$ and $g$ are analytic in $D$ (see \cite{Duren}).
Throughout this paper we use the decomposition $f=h+\bar g$.

Motivated by the work of Colonna on Bloch harmonic mappings \cite{Colonna}, Arbel\'aez,
Hern\'andez and Sierra \cite{AHS} introduced the notion of normality for harmonic mappings.

\begin{definition}\label{def:normalharm}
A harmonic mapping $f=h+\bar g$ in $\D$ is said to be \emph{normal} in $\D$ if it satisfies
the Lipschitz-type condition
\[
\sup_{\substack{z,w\in\D\\ z\neq w}}\frac{\chi(f(z),f(w))}{\rho_{h}(z,w)}<\infty,
\]
where $\chi(f(z),f(w))$ is the chordal distance, i.e.\ the Euclidean distance between the
stereographic projections of $f(z)$ and $f(w)$ on the Riemann sphere,
\[
\chi(f(z),f(w))=
\begin{cases}
\dfrac{|f(z)-f(w)|}{\sqrt{1+|f(z)|^{2}}\,\sqrt{1+|f(w)|^{2}}}, & f(z),f(w)\in\C,\\[2.2ex]
\dfrac{1}{\sqrt{1+|f(z)|^{2}}}, & f(w)=\infty,\\[2.2ex]
0, & f(z)=f(w)=\infty,
\end{cases}
\]
and $\rho_{h}(z,w)$ is the hyperbolic distance between $z$ and $w$,
\[
\rho_{h}(z,w)=\frac12\log\!\left(\frac{1+t}{1-t}\right),\qquad
t=\left|\frac{z-w}{1-\bar w z}\right|.
\]
\end{definition}

The following equivalent characterization is often more convenient.

\begin{definition}[{\cite[Proposition~2.1]{AHS}}]\label{def:normalharmeq}
A harmonic mapping $f=h+\bar g$ in $\D$ is \emph{normal} in $\D$ if
\begin{equation}\label{eq:normalharmeq}
\sup_{z\in\D}\,(1-|z|^{2})\,\frac{|h'(z)|+|g'(z)|}{1+|f(z)|^{2}}<\infty.
\end{equation}
\end{definition}

The theory of normal harmonic mappings has attracted significant attention. Deng,
Ponnusamy and Qiao \cite{DPQ} established harmonic analogues of several classical theorems
on normal meromorphic functions, including versions of Marty's theorem and the
Lohwater--Pommerenke rescaling theorem. Bharti and Thin \cite{BT} established a
Zalcman--Pang type rescaling result in the same setting. Following \cite{BT}, the
\emph{spherical derivative} of a harmonic mapping $f=h+\bar g$ in $\D$ is
\[
f^{\#}(z)=\frac{|h'(z)|+|g'(z)|}{1+|f(z)|^{2}},
\]
and, for $k\in\N$, the \emph{extended spherical derivative of order $k$} is
\begin{equation}\label{eq:extsph}
f^{\#(k)}(z):=\frac{|h^{(k)}(z)|+|g^{(k)}(z)|}{1+|f(z)|^{k+1}}.
\end{equation}

Motivated by the generalized Marty-type theorem of Li and Xie \cite{LiXie} and by the
normality criteria of Chen, Nevo and Pang \cite{CNP} involving lower bounds of differential
expressions, we first establish a new criterion for a harmonic mapping to be normal.

\begin{theorem}\label{thm:lowerbound}
Let $k$ be a positive integer and let $M>0$. Suppose that $f=h+\bar g$ is a non-constant
harmonic mapping in $\D$ satisfying
\begin{equation}\label{eq:lowerbound}
f^{\#(k)}(z)>M,\qquad z\in\D.
\end{equation}
Then $f$ is a normal harmonic mapping.
\end{theorem}

Choosing $k=1$ in Theorem~\ref{thm:lowerbound} recovers \cite[Theorem~1.4]{BT}. Moreover
Theorem~\ref{thm:lowerbound} may be regarded as a harmonic analogue, for higher-order
derivatives, of the result of Grahl and Nevo \cite{GrahlNevo}.

We next study weighted versions of normality. Our work is motivated by Aulaskari and
R\"atty\"a \cite{AR}, who introduced $\varphi$-normal meromorphic functions. A function
$\varphi\colon[0,1)\to(0,\infty)$ is said to be \emph{smoothly increasing} if
\[
\varphi(r)(1-r)\longrightarrow\infty\quad(r\to1^{-}),\qquad
\frac{\varphi\!\left(\left|a+\dfrac{z}{\varphi(|a|)}\right|\right)}{\varphi(|a|)}\longrightarrow 1
\quad(|a|\to1^{-}),
\]
the latter converges uniformly on compact subsets of $\C$. A typical example is
$\varphi(r)=(1-r)^{-\alpha}$, $\alpha\in(1,\infty)$.

If $\varphi$ is smoothly increasing then, without loss of generality, we shall assume
throughout that
\begin{equation}\label{eq:phinorm}
 \varphi(r)(1-r)\ge 1,\qquad 0\le r<1.
\end{equation}
Indeed, replacing $\varphi$ by $\max\{\varphi,1\}$ changes neither the smoothly increasing
property nor the notion of $\varphi$-normality below, since the two functions agree near
the boundary, where the relevant estimates take place.

Given such a $\varphi$, a meromorphic function in $\D$ is \emph{$\varphi$-normal} if
\[
\sup_{z\in\D}\frac{f^{\#}(z)}{\varphi(|z|)}<\infty.
\]
Following \cite{AR}, Bohra, Datt and Pal \cite{BDP} extended $\varphi$-normality to planar
harmonic mappings: a harmonic mapping $f=h+\bar g$ in $\D$ is \emph{$\varphi$-normal} if
\[
\sup_{z\in\D}\frac{f^{\#}(z)}{\varphi(|z|)}<\infty,
\]
where $f^{\#}$ is as in \eqref{eq:extsph} with $k=1$.

The theory of $\varphi$-normal harmonic mappings is largely unexplored. Our central tool is
the following Zalcman--Pang type rescaling lemma. Only the implication stated is required in
the sequel; see Remark~\ref{rem:converse} for a discussion of the converse.

\begin{theorem}\label{thm:rescale}
Let $\varphi\colon[0,1)\to(0,\infty)$ be a continuous smoothly increasing function and let
$\beta\in(-1,\infty)$. If a non-constant harmonic mapping $f=h+\bar g$ in $\D$ is
\emph{not} $\varphi$-normal, then there exist sequences $\{w_{n}\}\subset\D$ and
$\{\rho_{n}\}\subset(0,1)$, with $\rho_{n}\to0$ as $n\to\infty$, such that the functions
\[
F_{n}(\eta)=\left(\frac{\rho_{n}}{\varphi(|w_{n}|)}\right)^{\beta}
f\!\left(w_{n}+\frac{\rho_{n}\eta}{\varphi(|w_{n}|)}\right)
\]
converge locally uniformly in $\C$ to a non-constant harmonic mapping $F$ satisfying
$F^{\#}(\eta)\le F^{\#}(0)=1$.
\end{theorem}

As a consequence we obtain the following sufficient condition for $\varphi$-normality.

\begin{theorem}\label{thm:phisuff}
Let $k$ be a positive integer and let $\alpha>1$, $M>0$. Let
$\varphi\colon[0,1)\to(0,\infty)$ be a continuous smoothly increasing function. Suppose that
$f=h+\bar g$ is a non-constant harmonic mapping in $\D$ satisfying
\begin{equation}\label{eq:phisuff}
\frac{|h^{(k)}(z)|+|g^{(k)}(z)|}{1+|f(z)|^{\alpha}}>\frac{M}{\varphi(|z|)^{\alpha-1}},
\qquad z\in\D.
\end{equation}
Then $f$ is a $\varphi$-normal harmonic mapping in $\D$.
\end{theorem}

\begin{example}\label{ex:phisuff}
Let $\varphi(r)=(1-r)^{-2}$, $r\in[0,1)$, and let $f(z)=z+\tfrac12\bar z$, so that
$h(z)=z$ and $g(z)=\tfrac12 z$. Take $\alpha=\tfrac32$ and
$M=\tfrac32\big/\big(1+(\tfrac32)^{3/2}\big)$. Since $|h'(z)|+|g'(z)|=\tfrac32$ and
$|f(z)|\le\tfrac32|z|<\tfrac32$, we have for all $z\in\D$
\[
\frac{|h'(z)|+|g'(z)|}{1+|f(z)|^{3/2}}
=\frac{3/2}{1+|f(z)|^{3/2}}
>\frac{3/2}{1+(3/2)^{3/2}}
=M\ge M(1-|z|)=\frac{M}{\varphi(|z|)^{1/2}},
\]
so $f$ satisfies \eqref{eq:phisuff}. Moreover
\[
\frac{f^{\#}(z)}{\varphi(|z|)}=\frac{3/2}{1+|f(z)|^{2}}\,(1-|z|)^{2}\le\frac32,
\]
so $\sup\limits_{z\in\D}f^{\#}(z)/\varphi(|z|)<\infty$ and $f$ is $\varphi$-normal, as predicted by
Theorem~\ref{thm:phisuff}.
\end{example}

\begin{corollary}\label{cor:phisuff}
Let $k$ be a positive integer, $M>0$, and let $\varphi\colon[0,1)\to(0,\infty)$ be a
continuous smoothly increasing function. If $f=h+\bar g$ is a non-constant harmonic mapping
in $\D$ satisfying
\begin{equation}\label{eq:corphisuff}
f^{\#(k)}(z)>\frac{M}{\varphi(|z|)^{k}},\qquad z\in\D,
\end{equation}
then $f$ is $\varphi$-normal.
\end{corollary}

We next pass to families. Following Lohwater and Pommerenke
\cite[Theorem~1]{LohwaterPommerenke}, Zalcman \cite{Zalcman} established his rescaling lemma
for normal families of meromorphic functions, since then one of the most powerful tools in
the theory. We introduce the corresponding notion in the present setting.

\begin{definition}\label{def:phifamily}
A family $\mathcal F$ of harmonic mappings in $\D$ is said to be \emph{$\varphi$-normal} in
$\D$ if for each compact set $K\subset\D$ there exists $M>0$ such that
\[
\sup\left\{\frac{f^{\#}(z)}{\varphi(|z|)}:z\in K,\ f\in\mathcal F\right\}<M.
\]
\end{definition}

\begin{example}\label{ex:phifamily}
Let $\varphi(|z|)=(1-|z|)^{-2}$ and
$\mathcal F=\big\{f_{n}(z)=z+\tfrac1n\,\overline{z^{\,n}}:z\in\D\big\}$. Then
$h(z)=z$, $g_{n}(z)=\tfrac1n z^{n}$, so $|h'(z)|+|g_{n}'(z)|=1+|z|^{n-1}$ and
\[
f_{n}^{\#}(z)=\frac{1+|z|^{\,n-1}}{1+|f_{n}(z)|^{2}}.
\]
Hence for every compact $K\subset\D$,
\[
\sup_{\substack{z\in K\\ f_{n}\in\mathcal F}}\frac{f_{n}^{\#}(z)}{\varphi(|z|)}
=\sup_{\substack{z\in K\\ f_{n}\in\mathcal F}}
\frac{(1+|z|^{\,n-1})(1-|z|)^{2}}{1+|f_{n}(z)|^{2}}<2,
\]
so $\mathcal F$ is a $\varphi$-normal family.
\end{example}

\begin{theorem}\label{thm:rescalefamily}
Let $\varphi\colon[0,1)\to(0,\infty)$ be a continuous smoothly increasing function and let
$\beta\in(-1,\infty)$. If a family $\mathcal F$ of non-constant harmonic mappings in $\D$ is
\emph{not} $\varphi$-normal, then there exist sequences $\{w_{n}\}\subset\D$,
$\{f_{n}\}\subset\mathcal F$ and $\{\rho_{n}\}\subset(0,1)$, with $\rho_{n}\to0$, such that
\[
F_{n}(\eta)=\left(\frac{\rho_{n}}{\varphi(|w_{n}|)}\right)^{\beta}
f_{n}\!\left(w_{n}+\frac{\rho_{n}\eta}{\varphi(|w_{n}|)}\right)
\]
converge locally uniformly in $\C$ to a non-constant harmonic mapping $F$ satisfying
$F^{\#}(\eta)\le F^{\#}(0)=1$.
\end{theorem}

As an application we obtain a Lappan-type theorem. The novelty is that the test set $E$ may
be taken to have only \emph{three} points (rather than five), a saving made possible by
Lemma~\ref{lem:harmtwo} below.

\begin{theorem}\label{thm:lappanfamily}
Let $\varphi\colon[0,1)\to(0,\infty)$ be a continuous smoothly increasing function. Let
$\mathcal F$ be a family of sense-preserving harmonic mappings in $\D$ and let $E$ be a set
of three distinct complex numbers such that
\[
\sup\left\{\frac{f^{\#}(w)}{\varphi(|w|)}:w\in f^{-1}(E),\ f\in\mathcal F\right\}<\infty.
\]
Then $\mathcal F$ is $\varphi$-normal in $\D$.
\end{theorem}

\begin{theorem}\label{thm:lappansingle}
Let $f=h+\bar g$ be a non-constant sense-preserving harmonic mapping in $\D$ and let $E$ be
a set of three distinct complex numbers such that
\[
\sup\left\{\frac{f^{\#}(w)}{\varphi(|w|)}:w\in f^{-1}(E)\right\}<\infty.
\]
Then $f$ is a $\varphi$-normal harmonic mapping in $\D$.
\end{theorem}

\begin{remark}\label{rem:improveBDP}
Theorem~\ref{thm:lappansingle} improves \cite[Theorem~1.6]{BDP} by reducing the cardinality
of the test set from five to three. By the same argument, the cardinality of the set $E$ in
the Lappan-type criterion \cite[Theorem~1.6]{BT} can likewise be reduced to three.
\end{remark}

\section{Preliminary lemmas}

We collect the auxiliary results used in the proofs. The following is the harmonic analogue of Marty's theorem, due to Deng, Ponnusamy and Qiao \cite{DPQ}.

\begin{lemma}[{\cite[Lemma~1]{DPQ}}]\label{lem:marty}
A family $\mathcal F$ of harmonic mappings $f=h+\bar g$ in $\D$ is normal if
$\{f^{\#}(z):f\in\mathcal F\}$ is locally uniformly bounded.
\end{lemma}

\begin{remark}\label{rem:martyconverse}
A partial converse holds under the additional assumption $\Real\big(h(z)\overline{g(z)}\big)>0$;
see \cite[Lemma~2.1]{BDP}.
\end{remark}

\begin{lemma}[{\cite[Theorem~1.2]{BT}}]\label{lem:resnormal}
A non-constant harmonic mapping $f$ in $\D$ is normal if and only if there do \emph{not}
exist sequences $\{w_{n}\}\subset\D$ and $\{\rho_{n}\}\subset(0,\infty)$ with $\rho_{n}\to0$
such that
\[
F_{n}(\eta):=\rho_{n}^{\,\beta}\,f(w_{n}+\rho_{n}\eta)
\]
converges locally uniformly in $\C$ to a non-constant harmonic mapping $F$ satisfying
$F^{\#}(\eta)\le F^{\#}(0)=1$, where $\beta\in(-1,\infty)$ is arbitrary but fixed.
\end{lemma}

\begin{lemma}[{\cite[Theorem~1.3]{BDP}}]\label{lem:phizalcman}
Let $\varphi\colon[0,1)\to(0,\infty)$ be smoothly increasing. A non-constant harmonic
mapping $f=h+\bar g$ in $\D$ is \emph{not} $\varphi$-normal if and only if there exist
sequences $\{w_{n}\}\subset\D$ and $\{\rho_{n}\}\subset(0,1)$, with $\rho_{n}\to0$, such that
\[
F_{n}(\eta)=f\!\left(w_{n}+\frac{\rho_{n}\eta}{\varphi(|w_{n}|)}\right)
\]
converges locally uniformly in $\C$ to a non-constant harmonic mapping $F$.
\end{lemma}

\begin{lemma}[{\cite[p.~10]{Duren}}]\label{lem:hurwitz}
Let $\{f_{n}\}$ be a sequence of sense-preserving harmonic mappings in $\D$ converging
locally uniformly to a sense-preserving harmonic mapping $f$. Then $z_{0}\in\D$ is a zero of
$f$ if and only if $z_{0}$ is a cluster point of the zeros of $f_{n}$, $n\ge1$.
\end{lemma}

The next lemma is classical (see \cite{Hayman}); we include a proof for completeness.

\begin{lemma}\label{lem:entiretwo}
Let $f$ be a non-constant entire function. Then there are at most two values $a$ for which
all zeros of $f-a$ are multiple.
\end{lemma}

\begin{proof}
We use the second fundamental theorem of Nevanlinna. Suppose, to the contrary, that there
exist three distinct values $a_{1},a_{2},a_{3}\in\C$ such that all zeros of $f-a_{j}$
($j=1,2,3$) have multiplicity at least $2$. Since $f$ is entire, $\overline N(r,f)=0$, and
since every zero of $f-a_{j}$ is multiple,
$$\overline N\!\big(r,\tfrac{1}{f-a_{j}}\big)\le\tfrac12 N\!\big(r,\tfrac{1}{f-a_{j}}\big)
\le\tfrac12 T(r,f)+O(1).$$ Hence
\[
2T(r,f)\le \overline N(r,f)+\sum_{j=1}^{3}\overline N\!\left(r,\frac{1}{f-a_{j}}\right)+S(r,f)
\le \frac32\,T(r,f)+S(r,f),
\]
so that $\tfrac12 T(r,f)\le S(r,f)$, which is impossible. Hence the proof.
\end{proof}

The following sharpens \cite[Lemma~3]{DPQ}, where four values appear in place of two.

\begin{lemma}\label{lem:harmtwo}
Let $f=h+\bar g$ be a sense-preserving harmonic mapping in $\C$ with $g(0)=0$. Then there
are at most two values $a$ for which all zeros of $f-a$ are multiple.
\end{lemma}

\begin{proof}
Since $f$ is sense-preserving in $\C$, we have $|g'(z)/h'(z)|<1$ in $\C$; as $g'/h'$ is then
a bounded entire function, Liouville's theorem gives $\omega(z)=g'(z)/h'(z)\equiv c$ with
$c\in\C$, $|c|<1$. Integrating $g'=c\,h'$ with $g(0)=0$ yields $g(z)=c\big(h(z)-h(0)\big)$,
whence
\[
f(z)=h(z)+\overline{c\,h(z)-c\,h(0)}.
\]
As in \cite[Lemma~2.5]{BDP}, for any $a\in\C$ the equation $f(z)=a$ is equivalent to
\begin{equation}\label{eq:correspond}
h(z)=a^{*}:=\frac{a-\overline{c\,a}+\overline{c\,h(0)}-|c|^{2}h(0)}{1-|c|^{2}}.
\end{equation}
Moreover $f(z)-a=u(z)+\bar c\,\overline{u(z)}$ with $u=h-a^{*}$, so $z_{0}$ is a zero of
$f-a$ of order $p$ if and only if it is a zero of $h-a^{*}$ of order $p$; in particular the
multiplicities coincide. Since $a\mapsto a^{*}$ is a bijection of $\C$, and since by
Lemma~\ref{lem:entiretwo} there are at most two values $a^{*}$ for which every zero of
$h-a^{*}$ is multiple, there are at most two values $a$ for which every zero of $f-a$ is
multiple.
\end{proof}

%We shall also use the identity theorem for harmonic mappings: if a harmonic mapping vanishes
%on a non-empty open subset of a domain, it vanishes identically there.

\section{Proofs of the main results}

\begin{proof}[Proof of Theorem~\textup{\ref{thm:lowerbound}}]
Suppose, to the contrary, that $f$ is not normal. By Lemma~\ref{lem:resnormal} (with
$\beta=2$), there exist sequences $\{w_{n}\}\subset\D$ and $\{\rho_{n}\}\subset(0,1)$ with
$|w_{n}|\to1^{-}$ and $\rho_{n}\to0$ such that
\[
F_{n}(\eta)=\rho_{n}^{2}f(w_{n}+\rho_{n}\eta)
=\rho_{n}^{2}\big(h(w_{n}+\rho_{n}\eta)+\overline{g(w_{n}+\rho_{n}\eta)}\big)
\]
converges locally uniformly in $\C$ to a non-constant harmonic mapping
$F(\eta)=H(\eta)+\overline{G(\eta)}$ with $F^{\#}(\eta)\le1$ and $F^{\#}(0)=1$. Write
$H_{n}(\eta):=\rho_{n}^{2}h(w_{n}+\rho_{n}\eta)$ and
$G_{n}(\eta):=\rho_{n}^{2}g(w_{n}+\rho_{n}\eta)$. Since $F_{n}\to F$ locally uniformly, for
every positive integer $k$ we have $H_{n}^{(k)}\to H^{(k)}$ and $G_{n}^{(k)}\to G^{(k)}$
locally uniformly. As $F$ is non-constant there is $\eta_{0}$ with $F(\eta_{0})\neq0$.

Put $\zeta_{n}=w_{n}+\rho_{n}\eta_{0}$. Using
$H_{n}^{(k)}(\eta_{0})=\rho_{n}^{k+2}h^{(k)}(\zeta_{n})$,
$G_{n}^{(k)}(\eta_{0})=\rho_{n}^{k+2}g^{(k)}(\zeta_{n})$ and
$|F_{n}(\eta_{0})|^{k+1}=\rho_{n}^{2(k+1)}|f(\zeta_{n})|^{k+1}$, we estimate
\begin{align*}
|H_{n}^{(k)}(\eta_{0})|+|G_{n}^{(k)}(\eta_{0})|
&\ge\frac{|H_{n}^{(k)}(\eta_{0})|+|G_{n}^{(k)}(\eta_{0})|}{1+|F_{n}(\eta_{0})|^{k+1}}
=\frac{\rho_{n}^{k+2}\big(|h^{(k)}(\zeta_{n})|+|g^{(k)}(\zeta_{n})|\big)}
{1+\rho_{n}^{2(k+1)}|f(\zeta_{n})|^{k+1}}\\[0.4ex]
&=\frac{|h^{(k)}(\zeta_{n})|+|g^{(k)}(\zeta_{n})|}{\rho_{n}^{k}\,|f(\zeta_{n})|^{k+1}}\cdot
\frac{|F_{n}(\eta_{0})|^{k+1}}{1+|F_{n}(\eta_{0})|^{k+1}}\\[0.4ex]
&\ge\frac{|h^{(k)}(\zeta_{n})|+|g^{(k)}(\zeta_{n})|}{\rho_{n}^{k}\big(1+|f(\zeta_{n})|^{k+1}\big)}\cdot
\frac{|F_{n}(\eta_{0})|^{k+1}}{1+|F_{n}(\eta_{0})|^{k+1}}\\[0.4ex]
&>\frac{M}{\rho_{n}^{k}}\cdot\frac{|F_{n}(\eta_{0})|^{k+1}}{1+|F_{n}(\eta_{0})|^{k+1}},
\end{align*}
where the last inequality uses the hypothesis \eqref{eq:lowerbound}, i.e.\
$\big(|h^{(k)}|+|g^{(k)}|\big)/(1+|f|^{k+1})>M$. Since $\rho_{n}\to0$ and
$|F_{n}(\eta_{0})|\to|F(\eta_{0})|\neq0$, the right-hand side tends to $\infty$. But the
left-hand side converges to $|H^{(k)}(\eta_{0})|+|G^{(k)}(\eta_{0})|<\infty$, a
contradiction. Hence $f$ is normal.
\end{proof}

\begin{proof}[Proof of Theorem~\textup{\ref{thm:rescale}}]
Suppose $f$ is not $\varphi$-normal. Then there is a sequence $\{w_{n}'\}\subset\D$ with
$|w_{n}'|\to1^{-}$ and
\begin{equation}\label{eq:p1}
\frac{f^{\#}(w_{n}')}{\varphi(|w_{n}'|)}\longrightarrow\infty.
\end{equation}
By the continuity of $\varphi$ we may choose $r_{n}\in(0,1)$ with $|w_{n}'|<r_{n}<1$ and
\begin{equation}\label{eq:p2}
\frac{f^{\#}(w_{n}')}{\varphi\big(|w_{n}'|/r_{n}\big)}\longrightarrow\infty.
\end{equation}
For $t\in(0,1]$ and $|w|<r_{n}$ define
\[
\Gamma_{n}(t,w):=\frac{1}{\varphi\big(|w|/r_{n}\big)^{\,1+\beta}}\cdot
\frac{t^{\,1+\beta}\big(1+|f(w)|^{2}\big)f^{\#}(w)}
{1+\dfrac{t^{2\beta}|f(w)|^{2}}{\varphi(|w|/r_{n})^{2\beta}}}.
\]
Then $\Gamma_{n}$ is continuous on $(0,1]\times\{|w|<r_{n}\}$, and since $\beta>-1$,
\begin{equation}\label{eq:p3}
\lim_{t\to0}\Gamma_{n}(t,w)=0.
\end{equation}

\smallskip
\noindent\textbf{Claim 1.}
$\displaystyle \Gamma_{n}(t,w)\ge\frac{t^{\,1+|\beta|}f^{\#}(w)}{\varphi(|w|/r_{n})^{\,1+|\beta|}}.$

\smallskip
\noindent\emph{Case $\beta\ge0$.} By \eqref{eq:phinorm} we have $\varphi(|w|/r_{n})\ge1$,
and since $t\le1$ this gives $t^{2\beta}/\varphi(|w|/r_{n})^{2\beta}\le1$. Hence the
denominator is at most $1+|f(w)|^{2}$, and
\[
\Gamma_{n}(t,w)\ge\frac{1}{\varphi(|w|/r_{n})^{1+\beta}}\cdot
\frac{t^{1+\beta}\big(1+|f(w)|^{2}\big)f^{\#}(w)}{1+|f(w)|^{2}}
=\frac{t^{1+\beta}f^{\#}(w)}{\varphi(|w|/r_{n})^{1+\beta}}.
\]
\emph{Case $\beta<0$.} Then $|\beta|=-\beta>0$, and the same computation with $|\beta|$ in
place of $\beta$ yields the claim. This proves Claim~1.

\smallskip
By \eqref{eq:p2} and Claim~1,
\begin{equation}\label{eq:p4}
\Gamma_{n}(1,w_{n}')\ge\frac{f^{\#}(w_{n}')}{\varphi(|w_{n}'|/r_{n})^{\,1+|\beta|}}
\longrightarrow\infty.
\end{equation}
Thus $\sup_{|w|<r_{n}}\Gamma_{n}(1,w)\ge\Gamma_{n}(1,w_{n}')>1$ for large $n$, while by
\eqref{eq:p3} $\sup_{|w|<r_{n}}\Gamma_{n}(t,w)<1$ for $t$ small. By continuity there exist
$t_{n}\in(0,1)$ and $w_{n}$ with $|w_{n}|<r_{n}$ such that
\begin{equation}\label{eq:p5}
\sup_{|w|<r_{n}}\Gamma_{n}(t_{n},w)=\Gamma_{n}(t_{n},w_{n})=1.
\end{equation}
Combining \eqref{eq:p5} with Claim~1,
\[
1=\Gamma_{n}(t_{n},w_{n})\ge\Gamma_{n}(t_{n},w_{n}')\ge
\frac{t_{n}^{\,1+|\beta|}f^{\#}(w_{n}')}{\varphi(|w_{n}'|/r_{n})^{\,1+|\beta|}},
\]
which together with \eqref{eq:p4} forces $t_{n}\to0$.

Set $\rho_{n}=\dfrac{t_{n}\varphi(|w_{n}|)}{\varphi(|w_{n}|/r_{n})}$. Since $\varphi$ is
increasing and $|w_{n}|<|w_{n}|/r_{n}$, we have $\rho_{n}<t_{n}\to0$. Define
\[
F_{n}(\eta)=\left(\frac{\rho_{n}}{\varphi(|w_{n}|)}\right)^{\beta}
f\!\left(w_{n}+\frac{\rho_{n}\eta}{\varphi(|w_{n}|)}\right),
\qquad |\eta|<R_{n}=\frac{(1-|w_{n}|)\varphi(|w_{n}|)}{\rho_{n}}.
\]
A direct computation gives
\begin{equation}\label{eq:p7}
F_{n}^{\#}(\eta)=\left(\frac{\rho_{n}}{\varphi(|w_{n}|)}\right)^{1+\beta}
\frac{\Big(1+\big|f(w_{n}+\tfrac{\rho_{n}\eta}{\varphi(|w_{n}|)})\big|^{2}\Big)
f^{\#}\!\big(w_{n}+\tfrac{\rho_{n}\eta}{\varphi(|w_{n}|)}\big)}
{1+\big(\tfrac{\rho_{n}}{\varphi(|w_{n}|)}\big)^{2\beta}
\big|f(w_{n}+\tfrac{\rho_{n}\eta}{\varphi(|w_{n}|)})\big|^{2}},
\end{equation}
and in particular, by \eqref{eq:p5},
\begin{equation}\label{eq:p8}
F_{n}^{\#}(0)=\left(\frac{\rho_{n}}{\varphi(|w_{n}|)}\right)^{1+\beta}
\frac{(1+|f(w_{n})|^{2})f^{\#}(w_{n})}
{1+\big(\tfrac{\rho_{n}}{\varphi(|w_{n}|)}\big)^{2\beta}|f(w_{n})|^{2}}
=\Gamma_{n}(t_{n},w_{n})=1.
\end{equation}

\smallskip
\noindent\textbf{Claim 2.} $\{F_{n}\}$ is normal.

\smallskip
Fix $R>0$ and let $|\eta|\le R$. By the smoothly increasing property,
\[
\frac{\varphi\!\Big(\big|\tfrac{w_{n}}{r_{n}}+\tfrac{\rho_{n}\eta}
{\varphi(|w_{n}|/r_{n})}\big|\Big)}{\varphi(|w_{n}|/r_{n})}\longrightarrow1
\]
uniformly for $|\eta|\le R$, so there is a sequence $\delta_{n}\to0^{+}$ with, for all large
$n$,
\begin{equation}\label{eq:p9}
\frac{(1-\delta_{n})t_{n}}
{\varphi\!\big(\big|\tfrac{w_{n}}{r_{n}}+\tfrac{\rho_{n}\eta}{\varphi(|w_{n}|/r_{n})}\big|\big)}
\le\frac{\rho_{n}}{\varphi(|w_{n}|)}
\le\frac{(1+\delta_{n})t_{n}}
{\varphi\!\big(\big|\tfrac{w_{n}}{r_{n}}+\tfrac{\rho_{n}\eta}{\varphi(|w_{n}|/r_{n})}\big|\big)}.
\end{equation}
Substituting \eqref{eq:p9} into \eqref{eq:p7} and bounding the denominator from below,
\[
F_{n}^{\#}(\eta)\le\frac{(1+\delta_{n})^{1+\beta}}{(1-\delta_{n})^{2\beta}}
\longrightarrow1\qquad(n\to\infty),
\]
uniformly for $|\eta|\le R$. Hence $\{F_{n}^{\#}\}$ is locally uniformly bounded and, by
Lemma~\ref{lem:marty}, the family $\{F_{n}\}$ is normal. This proves Claim~2.

\smallskip
Passing to a subsequence, $F_{n}\to F$ locally uniformly in $\C$ for some harmonic mapping
$F$. By \eqref{eq:p8} and Claim~2, $F^{\#}(\eta)\le1$ and $F^{\#}(0)=1$; in particular
$F^{\#}(0)\neq0$, so $F$ is non-constant. This completes the proof.
\end{proof}

\begin{remark}\label{rem:converse}
Only the implication proved in Theorem~\ref{thm:rescale} is used below. The converse
(``existence of a non-constant rescaled limit implies non-$\varphi$-normality'') holds for
$\beta=0$ and is precisely Lemma~\ref{lem:phizalcman}; for $\beta\ge1$ it may fail, since the bound of the rescaled spherical derivative carries a factor of order
$\rho_{n}^{\,1-\beta}\varphi(|w_{n}|)^{\beta}$ which may tend to $\infty.$  However, the converse holds for $\beta\leq 0.$ 
\end{remark}

\begin{proof}[Proof of Theorem~\textup{\ref{thm:phisuff}}]
Suppose, to the contrary, that $f$ is not $\varphi$-normal. Applying
Theorem~\ref{thm:rescale} with $\beta=\dfrac{k}{\alpha-1}+1\;(>1)$, there exist sequences
$\{w_{n}\}\subset\D$ and $\{\rho_{n}\}\subset(0,1)$ with $\rho_{n}\to0$ such that
\[
F_{n}(\eta)=\left(\frac{\rho_{n}}{\varphi(|w_{n}|)}\right)^{\frac{k}{\alpha-1}+1}
f\!\left(w_{n}+\frac{\rho_{n}\eta}{\varphi(|w_{n}|)}\right)
=H_{n}(\eta)+\overline{G_{n}(\eta)}
\]
converges locally uniformly in $\C$ to a non-constant harmonic mapping
$F=H+\overline G$ with $F^{\#}\le1$ and $F^{\#}(0)=1$, where
\[
H_{n}(\eta)=\left(\frac{\rho_{n}}{\varphi(|w_{n}|)}\right)^{\frac{k}{\alpha-1}+1}
h\!\left(w_{n}+\frac{\rho_{n}\eta}{\varphi(|w_{n}|)}\right),\quad
G_{n}(\eta)=\left(\frac{\rho_{n}}{\varphi(|w_{n}|)}\right)^{\frac{k}{\alpha-1}+1}
g\!\left(w_{n}+\frac{\rho_{n}\eta}{\varphi(|w_{n}|)}\right).
\]
Then $H_{n}^{(k)}\to H^{(k)}$ and $G_{n}^{(k)}\to G^{(k)}$ locally uniformly, and there is
$\eta_{0}$ with $F(\eta_{0})\neq0$. Put $w_{n}^{*}=w_{n}+\tfrac{\rho_{n}}{\varphi(|w_{n}|)}\eta_{0}$.
Writing $c_{n}=\rho_{n}/\varphi(|w_{n}|)$ and using
$H_{n}^{(k)}(\eta_{0})=c_{n}^{\,\frac{k}{\alpha-1}+1+k}h^{(k)}(w_{n}^{*})$ (similarly for
$G_{n}$) and $|F_{n}(\eta_{0})|^{\alpha}=c_{n}^{\,(\frac{k}{\alpha-1}+1)\alpha}|f(w_{n}^{*})|^{\alpha}$,
\begin{align*}
|H_{n}^{(k)}(\eta_{0})|+|G_{n}^{(k)}(\eta_{0})|
&\ge\frac{c_{n}^{\,\frac{k}{\alpha-1}+1+k}\big(|h^{(k)}(w_{n}^{*})|+|g^{(k)}(w_{n}^{*})|\big)}
{1+|F_{n}(\eta_{0})|^{\alpha}}\\[0.4ex]
&=\frac{|h^{(k)}(w_{n}^{*})|+|g^{(k)}(w_{n}^{*})|}{c_{n}^{\,\alpha-1}\,|f(w_{n}^{*})|^{\alpha}}\cdot
\frac{|F_{n}(\eta_{0})|^{\alpha}}{1+|F_{n}(\eta_{0})|^{\alpha}}\\[0.4ex]
&\ge\frac{|h^{(k)}(w_{n}^{*})|+|g^{(k)}(w_{n}^{*})|}
{c_{n}^{\,\alpha-1}\big(1+|f(w_{n}^{*})|^{\alpha}\big)}\cdot
\frac{|F_{n}(\eta_{0})|^{\alpha}}{1+|F_{n}(\eta_{0})|^{\alpha}}\\[0.4ex]
&>\frac{M}{\varphi(|w_{n}^{*}|)^{\alpha-1}\,c_{n}^{\,\alpha-1}}\cdot
\frac{|F_{n}(\eta_{0})|^{\alpha}}{1+|F_{n}(\eta_{0})|^{\alpha}}\\[0.4ex]
&=\frac{M}{\big(\varphi(|w_{n}^{*}|)/\varphi(|w_{n}|)\big)^{\alpha-1}\rho_{n}^{\,\alpha-1}}\cdot
\frac{|F_{n}(\eta_{0})|^{\alpha}}{1+|F_{n}(\eta_{0})|^{\alpha}},
\end{align*}
where we used the hypothesis \eqref{eq:phisuff} at $z=w_{n}^{*}$ and
$c_{n}^{\,\alpha-1}\varphi(|w_{n}^{*}|)^{\alpha-1}
=\rho_{n}^{\,\alpha-1}\big(\varphi(|w_{n}^{*}|)/\varphi(|w_{n}|)\big)^{\alpha-1}$. Since
$\varphi(|w_{n}^{*}|)/\varphi(|w_{n}|)\to1$, $\rho_{n}\to0$ and
$|F_{n}(\eta_{0})|\to|F(\eta_{0})|\neq0$, the right-hand side tends to $\infty$, while the
left-hand side tends to $|H^{(k)}(\eta_{0})|+|G^{(k)}(\eta_{0})|<\infty$, a contradiction.
Hence $f$ is $\varphi$-normal.
\end{proof}

\begin{proof}[Proof of Corollary~\textup{\ref{cor:phisuff}}]
Apply Theorem~\ref{thm:phisuff} with $\alpha=k+1$: then $\alpha-1=k$ and
\eqref{eq:corphisuff} is exactly \eqref{eq:phisuff}.
\end{proof}

\begin{proof}[Proof of Theorem~\textup{\ref{thm:rescalefamily}}]
The argument is identical to that of Theorem~\ref{thm:rescale}, with $f$ replaced by the
members $f_{n}$ of $\mathcal F$. Concretely, since $\mathcal F$ is not $\varphi$-normal there
exist $r'\in(0,1)$, $\{w_{n}'\}\subset\{|w|\le r'\}$ and $\{f_{n}\}\subset\mathcal F$ with
$f_{n}^{\#}(w_{n}')/\varphi(|w_{n}'|)\to\infty$; choosing $r_{n}\in(0,1)$ with
$|w_{n}'|<r_{n}<r'$ and replacing $f^{\#}$, $f$ by $f_{n}^{\#}$, $f_{n}$ throughout the
definition of $\Gamma_{n}$, Claims~1 and~2 and the construction of $w_{n},\rho_{n}$ go
through verbatim, producing the asserted limit $F$ with $F^{\#}\le F^{\#}(0)=1$.
\end{proof}

\begin{proof}[Proof of Theorem~\textup{\ref{thm:lappanfamily}}]
Suppose $\mathcal F$ is not $\varphi$-normal in $\D$. By Theorem~\ref{thm:rescalefamily}
(with $\beta=0$) there exist $\{w_{n}\}$, $\{\rho_{n}\}\subset(0,1)$ with $\rho_{n}\to0$ and
$\{f_{n}\}\subset\mathcal F$ such that
\[
F_{n}(\eta)=f_{n}\!\left(w_{n}+\frac{\rho_{n}}{\varphi(|w_{n}|)}\eta\right)\longrightarrow F(\eta)
\]
locally uniformly in $\C$, where $F$ is a non-constant sense-preserving harmonic mapping;
since $g(0)=0$ is a value-independent normalization, Lemma~\ref{lem:harmtwo} applies to $F$.

Let $b\in\C$ be such that $F-b$ has a simple zero at some $\eta_{0}$, i.e.\
$F^{\#}(\eta_{0})\neq0$. By Lemma~\ref{lem:hurwitz} there exist $\eta_{n}\to\eta_{0}$ with
$F_{n}(\eta_{n})=f_{n}(w_{n}^{*})=b$ for all large $n$, where
$w_{n}^{*}=w_{n}+\tfrac{\rho_{n}}{\varphi(|w_{n}|)}\eta_{n}$. Since $F_{n}\to F$ locally
uniformly,
\[
F_{n}^{\#}(\eta_{n})=\frac{\rho_{n}}{\varphi(|w_{n}|)}\,f_{n}^{\#}(w_{n}^{*})
\longrightarrow F^{\#}(\eta_{0})\neq0,
\]
and therefore, using $\varphi(|w_{n}^{*}|)/\varphi(|w_{n}|)\to1$ and $\rho_{n}\to0$,
\[
\frac{f_{n}^{\#}(w_{n}^{*})}{\varphi(|w_{n}^{*}|)}
=\frac{\varphi(|w_{n}|)}{\rho_{n}}\cdot\frac{F_{n}^{\#}(\eta_{n})}{\varphi(|w_{n}^{*}|)}
\longrightarrow\infty.
\]
Thus, whenever $b$ is a simple value of $F$,
\begin{equation}\label{eq:lappanblow}
\sup_{w\in f_{n}^{-1}(\{b\})}\frac{f_{n}^{\#}(w)}{\varphi(|w|)}=\infty.
\end{equation}
By Lemma~\ref{lem:harmtwo} there are at most two values $b$ for which every zero of $F-b$ is
multiple. Hence, among the three distinct numbers of $E$, at least one, say $b_{0}$, is a
simple value of $F$, and \eqref{eq:lappanblow} for $b_{0}$ contradicts the hypothesis
\[
\sup\left\{\frac{f^{\#}(w)}{\varphi(|w|)}:w\in f^{-1}(E),\ f\in\mathcal F\right\}<\infty.
\]
Therefore $\mathcal F$ is $\varphi$-normal in $\D$.
\end{proof}

\begin{proof}[Proof of Theorem~\textup{\ref{thm:lappansingle}}]
The proof is the single-mapping specialization of that of
Theorem~\ref{thm:lappanfamily}, using Lemma~\ref{lem:phizalcman} (the $\beta=0$ rescaling
for a single mapping) in place of Theorem~\ref{thm:rescalefamily}.
\end{proof}

\section*{Acknowledgements}
The second author gratefully acknowledges the Department of Science and Technology (DST),
Government of India, for financial support through the DST-INSPIRE Fellowship\\
(No.\ DST/INSPIRE/03/2022/005759: IF~220689).

\medskip
\noindent\textbf{Conflict of interest.} The authors declare that there are no conflicts of
interest regarding the publication of this paper.

\medskip
\noindent\textbf{Data availability.} Data sharing is not applicable, as the article is
purely theoretical.

\bigskip
\noindent
Kuldeep Singh Charak, Department of Mathematics, University of Jammu, Jammu-180006, India.\\
\emph{Email:} \texttt{kscharak7@rediffmail.com}

\medskip
\noindent
Pratiksha, Department of Mathematics, University of Jammu, Jammu-180006, India.\\
\emph{Email:} \texttt{pratikshapakhetra1999@gmail.com}

\medskip
\noindent
Nikhil Bharti, Higher Education Department, Government of Jammu and Kashmir, India;\\
Department of Mathematics, Government Degree College Basohli, Basohli-184201, Kathua,
Jammu and Kashmir, India.\\
\emph{Email:} \texttt{nikhilbharti94@gmail.com}

\end{document}